\documentclass{article}
\usepackage[T1]{fontenc}
\usepackage[utf8]{inputenc}
\usepackage{amsthm}
\usepackage{amssymb}
\usepackage{amsmath}
\usepackage{graphicx}
\usepackage{float}
\usepackage{caption}
\usepackage{url}
\usepackage{hyperref}
\usepackage{enumitem}
\usepackage{mathtools}

\newtheorem{lemma}{Lemma}

\newtheorem{prop}{Proposition}
\theoremstyle{definition}

\theoremstyle{remark}

\author{M.~Hellus
\and A.~Rechenauer
\and R.~Waldi}
\title{Note on a question of Wilf}
\begin{document}
\maketitle
\begin{abstract}
Let $S$ be a numerical semigroup with Frobenius number $f$, genus $g$ and embedding dimension $e$. 
In 1978 Wilf asked the question, whether $\frac{f+1-g}{f+1}\geq\frac1e$. As is well known, this holds in the cases $e=2$ and $e=3$. For $e\geq4$, we derive from results of Zhai \cite{Zhai} the following (substantially weaker) lower bound
\[\frac{f+1-g}{f+1}>\left(\frac{2N+1}{(2N+2)(e-2)}\right)^e\text{ with }\lfloor N\rfloor=104978\,.\]
To the best of our knowledge this is the first explicit lower bound for $\frac{f+1-g}{f+1}$ in terms of the embedding dimension.
\end{abstract}
A \emph{numerical semigroup} is an additively closed subset $S$ of $\mathbb Z_{\geq0}$ with $0\in S$ and only finitely many positive integers outside $S$, the so-called \emph{gaps} of $S$. The number of gaps is called the \emph{genus} $g(S)$. The set $E(S)=S^*\setminus(S^*+S^*)$, where $S^*= S\setminus\{0\}$, is the unique minimal system of generators of $S$. They are called the \emph{atoms} of $S$; their number $e(S)$ is the \emph{embedding dimension} of $S$.

From now on we assume that $S$ is different from $\mathbb Z_{\geq0}$. The largest gap $f(S)$ is called the \emph{Frobenius} number of $S$

In 1978, Wilf \cite{Wilf} asked the question: Is it always true that
\begin{equation}\label{formula_Wilf_question}\frac{f(S)+1-g(S)}{f(S)+1}\geq\frac1{e(S)}\ \text{?}\end{equation}
As is well known, this holds for $e(S)=2$, see \cite{Sylvester}, and e(S)=3, see \cite{Froeberg}. For short, we call the left hand side of (\ref{formula_Wilf_question}) the \emph{density} $d(S)$ of $S$ in the interval $[0,f(S)]$.

In the general case, Zhai \cite{Zhai} has given the following asymptotic answer to Wilf's question:

\vspace{.3cm}
\noindent\textbf{\cite[Theorem 2]{Zhai}} Given $\varepsilon>0$ real and $e\geq2$ integer, for all but finitely many numerical semigroups with $e(S)=e$ we have
\begin{equation}d(S)>\frac1e-\varepsilon\,.\end{equation}
\qed

From this, one gets immediately: For each $e\geq2$ there exists a positive real number $i(e)$ such that
\[d(S)>i(e)\text{ for all }S\text{ with }e(S)=e\,.\]
Looking on Zhai's paper more carefully, one may explicitely specify such a function. To this end we will use the auxiliary function
\[F(x,e)\coloneqq \frac1{(2x+2)e}-\left(\frac{2x+1}{(2x+2)(e-2)}\right)^e\text{, for }x\geq1\text{ real and }e\geq8\text{ integer.}\]
The following lemma is easy to see:
\begin{lemma}\label{lemma1} By Elementary Calculus we get
\begin{enumerate}
\item $F(x,8)$ is strictly decreasing with (unique) zero $N$, where $\lfloor N\rfloor=104978$. 
\end{enumerate}
Induction on $e$ shows that
\begin{enumerate}
\item[b)] $F(N,e)\geq0$ for $e\geq8$.
\end{enumerate}
\end{lemma}\qed
\begin{prop}Let $S$ be a numerical semigroup of embedding dimension $e\geq4$. Let $N$ be as above, $\lfloor N\rfloor=104978$. Then
\begin{enumerate}
\item
\[d(S)>\frac{8-e}{6e}\text{ in case }4\leq e\leq7\,.\]
\item
\[d(S)>\left(\frac{2N+1}{(2N+2)(e-2)}\right)^e\text{ in case }e\geq8\,.\]
\end{enumerate}
\end{prop}
\begin{proof} Let $m(S)$ the smallest atom, called the \emph{multiplicity} of $S$.
\begin{enumerate}
\item According to \cite[Theorem 1]{Zhai} we have
\begin{equation}\label{drei}d(S)\geq\frac1e-\frac{(m(S)-1)(e-2)}{(f(S)+1)\cdot2e}\,.\end{equation}
Since, according to Eliahou \cite[Proposition 3.13 and Corollary 6.5]{Eliahou}, Wilf holds in case $f(S)+1\leq3m(S)$, as well as $\frac1e>\frac{8-e}{6e}$ for $e\geq4$, we may assume that $m(S)<\frac{f(S)+1}3$. Substituting $m(S)-1$ by $\frac{f(S)+1}3$ in (\ref{drei}), we get the desired inequality.
\end{enumerate}
\end{proof}
As for b), more generally we shall show
\begin{prop}\label{prop2}Given an integer $e\geq8$, let $\Sigma(e)$ be the set of all numerical semigroups with embedding dimension $e$. Let $N$ be the (unique) zero of $F(x,8)$, $\lfloor N\rfloor=104978$. Then
\begin{enumerate}
\item For all but finitely many $S$ from $\Sigma(e)$ we have
\[d(S)>\frac1{(2N+2)\cdot e}\geq\left(\frac{2N+1}{(2N+2)(e-2)}\right)^e\,.\]

\item For all $S$ from $\Sigma(e)$, anyhow we have
\[d(S)>\left(\frac{2N+1}{(2N+2)(e-2)}\right)^e\,.\]
\end{enumerate}
\end{prop}\qed

\noindent\textbf{Preliminary consideration.} In \cite[Lemma 4 and its proof]{Zhai} it is shown:

\vspace{.3cm}
\begin{lemma}\label{lemma2} (Zhai) Given $\delta>0$ arbitrarily, for $S$ from $\Sigma(e)$ we have:

\begin{enumerate}
\item If $\frac{m(S)}{f(S)+1}>\delta$, then $f(S)+1\leq\left(\frac2\delta\right)^e$.
\item In $\Sigma(e)$ there are only finitely many $S$ such that $\frac{m(S)}{f(S)+1}>\delta$.
\end{enumerate}
\end{lemma}
\qed

\noindent\textit{Proof of Proposition \ref{prop2}, with the help of Lemma \ref{lemma2}.} In Lemma \ref{lemma2}, choose $\delta=\frac{2N+1}{(N+1)(e-2)}$.
\begin{enumerate}
\item According to Lemma \ref{lemma2}\,b), $\frac{m(S)}{f(S)+1}\leq\delta$ for almost all $S$ from $\Sigma(e)$.

Finally, under the assumption $\frac{m(S)}{f(S)+1}\leq\delta$, by \cite[Theorem 1]{Zhai} we get
\[d(S)\geq\frac1e-\frac{(m(S)-1)(e-2)}{(f(S)+1)2e}>\frac1e-\frac{m(S)(e-2)}{(f(S)+1)2e}\geq\frac1e-\delta\cdot\frac{e-2}{2e}=\frac1{(2N+2)e}\,.\]
The second claimed inequality follows from Lemma \ref{lemma1}\,b).
\item Now let $\frac{m(S)}{f(S)+1}>\delta$ (which, by Lemma \ref{lemma2}\,b), happens only finitely many times).

In case $f(S)+1-g(S)=1$ we have $f(S)+1=e$, hence
\[d(S)=\frac1{f(S)+1}=\frac1e>\frac1{(2N+2)e}\geq\left(\frac{2N+1}{(2N+2)(e-2)}\right)^e\text{\,, by Lemma \ref{lemma1}\,b).}\]
Finally, let $\frac{m(S)}{f(S)+1}>\delta$ and $f(S)+1-g(S)>1$. By Lemma \ref{lemma2}\,a) we get
\[d(S)>\frac1{f(S)+1}\geq\left(\frac\delta2\right)^e =\left(\frac{2N+1}{(2N+2)(e-2)}\right)^e\,.\]
\end{enumerate}
\qed

\noindent\textbf{Side note}. By Lemma \ref{lemma1}\,.a), of all possible values for which the above proof works, our choice of $N$ leads to the best estimation for $d(S)$.

\end{document}